\documentclass[10pt]{article}
\usepackage{amsfonts,color}
 \usepackage{amsmath,amssymb}
 \usepackage{epsfig}
\usepackage{lscape}
\usepackage{amssymb}
\usepackage{graphicx}
 \usepackage{footnote}
\usepackage[utf8]{inputenc}	
\usepackage{mathtools}
\usepackage{amsthm,amssymb,wasysym}
\usepackage{amsfonts}
\usepackage[english]{babel}

\usepackage{bbm}

\newcommand{\id}{{\boldsymbol{\mathbbm{1}}}}

\newcommand{\Chi}{\raisebox{0.5ex}{\mbox{{\Large $\chi$}}}}
\newcommand{\Partial}{\raisebox{0ex}{\mbox{{\large $\partial$}}}}
\allowdisplaybreaks[1]

\makeindex

 \newtheorem{theorem}{Theorem}[section]
 \newtheorem{lemma}[theorem]{Lemma}
  \newtheorem{postulate}[theorem]{Postulate}

 \newcommand{\R}{\mathbb{R}}

 \newcommand{\tel}[1]{\frac{1}{#1}}

 \DeclareMathOperator{\Sym}{Sym}

 \DeclareMathOperator{\dev}{dev}

\def\barr{\begin{array}}
\def\tr{\textrm{tr}}

\def\dd{\displaystyle}

\def\barr{\begin{array}}
\def\earr{\end{array}}
\def\bec#1{\begin{equation}\label{#1}}
\def\becn{\begin{equation*}}
\def\endec{\end{equation}}
\def\endecn{\end{equation*}}

\makeatletter
\let\@fnsymbol\@arabic
\makeatother

\begin{document}

\title{\vspace*{0cm}Loss of ellipticity for non-coaxial plastic deformations in additive logarithmic finite strain plasticity}
\author{
Patrizio Neff\thanks{Corresponding author: Patrizio Neff,  \ \ Head of Lehrstuhl f\"{u}r Nichtlineare Analysis und Modellierung, Fakult\"{a}t f\"{u}r
Mathematik, Universit\"{a}t Duisburg-Essen,  Thea-Leymann Str. 9, 45127 Essen, Germany, email: patrizio.neff@uni-due.de}\quad
and \quad
Ionel-Dumitrel Ghiba\thanks{Ionel-Dumitrel Ghiba, \ \ \ \ Lehrstuhl f\"{u}r Nichtlineare Analysis und Modellierung, Fakult\"{a}t f\"{u}r Mathematik,
Universit\"{a}t Duisburg-Essen, Thea-Leymann Str. 9, 45127 Essen, Germany;  Alexandru Ioan Cuza University of Ia\c si, Department of Mathematics,  Blvd.
Carol I, no. 11, 700506 Ia\c si,
Romania; and  Octav Mayer Institute of Mathematics of the
Romanian Academy, Ia\c si Branch,  700505 Ia\c si, email: dumitrel.ghiba@uni-due.de, dumitrel.ghiba@uaic.ro}}

\maketitle

\begin{abstract}
In this paper we consider the additive logarithmic finite strain plasticity formulation from the view point of loss of ellipticity in elastic unloading. We prove that even if an elastic energy $F\mapsto W(F)=\widehat{W}(\log U)$ defined in terms of logarithmic strain $\log U$, where $U=\sqrt{F^T\, F}$, happens to be everywhere rank-one convex as a function of $F$, the new function $F\mapsto \widetilde{W}(F)=\widehat{W}(\log U-\log U_p)$ need not remain rank-one convex at some given plastic stretch $U_p$ (viz. $E_p^{\log}:=\log U_p$). This is in complete contrast to multiplicative plasticity (and infinitesimal plasticity) in which $F\mapsto W(F\, F_p^{-1})$ remains rank-one convex at every plastic distortion $F_p$ if $F\mapsto W(F)$ is rank-one convex ($\nabla u \mapsto \|{\rm sym} \nabla u-\varepsilon_p\|^2$  remains convex). We show this {disturbing}  feature of the additive logarithmic plasticity model  with the help of a recently introduced  family of exponentiated Hencky energies.

\medskip

\end{abstract}

\section{Introduction}

Geometrically nonlinear plasticity is  a field of intensive ongoing research. There exist several fundamentally different approaches which reduce, more or less, to the infinitesimal  model based on the additive split of the infinitesimal strain tensor $\varepsilon={\rm sym}\nabla u$ into symmetric  elastic and plastic strains \cite{neff2009notes,Neff_Wieners,Neff01d,Neff_techmech07} $\varepsilon=\varepsilon_e+\varepsilon_p$. We assume the reader to be familiar with the general framework of finite strain plasticity models.

The involved nonlinearities make it difficult, both from an analysis and algorithmic point of view to obtain definitive results. There exists, however, one well known and much used methodology to reduce the algorithmic complexity dramatically. It is based on the matrix-logarithm and the introduction of a so called  plastic metric $C_p\in{\rm PSym}(3)$ together with an additive ansatz for elastic strains $E_e=\log C-\log C_p$. We refer to these models as additive logarithmic\footnote{We refrain completely from discussing hypoplastic models based on the so-called logarithmic rate, see e.g. \cite{bruhns1999self,xiao2000consistent} which also involve the Hencky strain.}. Within these, basically the small strain linearized framework is simply lifted to the geometrically nonlinear setting through the properties of the logarithm. Thus, frame-indifference, thermodynamical admissibility, plastic volume constraint, associative flow rule, principle of maximal dissipation  etc. are all easily satisfied.

In this paper, however,  we want to exhibit a major drawback of this approach which makes the additive logarithmic ansatz in our view inadmissible in those cases where large plastic deformations need to be considered, as are encountered e.g. in elastic spring back processes in the automobile industry.  Without loss of generality, we will concentrate our exposition to the  completely isotropic setting, in which most prominently the quadratic Hencky-logarithmic strain energy appears.

\subsection{Notation}\label{auxnot}

 For $a,b\in\R^n$ we let $\langle {a},{b}\rangle_{\R^n}$  denote the scalar product on $\R^n$ with
associated vector norm $\|a\|_{\R^n}^2=\langle {a},{a}\rangle_{\R^n}$.
We denote by $\R^{n\times n}$ the set of real $n\times n$ second order tensors, written with
capital letters.
The standard Euclidean scalar product on $\R^{n\times n}$ is given by
$\langle {X},{Y}\rangle_{\R^{n\times n}}=\tr{(X Y^T)}$, and thus the Frobenius tensor norm is
$\|{X}\|^2=\langle {X},{X}\rangle_{\R^{n\times n}}$. In the following we do not adopt any summation convention and we omit {the   subscript} $\R^{n\times
n}$ in writing the Frobenius tensor norm.    The identity tensor on $\R^{n\times n}$ will be denoted by $\id$, so that
$\tr{(X)}=\langle {X},{\id}\rangle$, while $\dev_n X=X-\frac{1}{n}\, \tr(X)\cdot\id$ is the n-dimensional deviatoric part of a second
order tensor $X{\in\R^{n\times n}}$.  We let $\Sym(n)$ and ${\rm PSym}(n)$ denote the symmetric and positive definite symmetric tensors respectively. We adopt
the usual abbreviations of Lie-group theory, i.e.,
${\rm GL}(n):=\{X\in\R^{n\times n}\;|\det{X}\neq 0\}$ denotes the general linear group,
${\rm SL}(n):=\{X\in {\rm GL}(n)\;|\det{X}=1\},\;
\mathrm{O}(n):=\{X\in {\rm GL}(n)\;|\;X^TX=\id\},\;{\rm SO}(n):=\{X\in {\rm GL}(n,\R)\;| X^T X=\id,\;\det{X}=1\}$, ${\rm GL}^+(n):=\{X\in\R^{n\times
n}\;|\det{X}>0\}$  is the group of invertible matrices with positive determinant,
 $\mathfrak{so}(3):=\{X\in\mathbb{R}^{3\times3}\;|X^T=-X\}$ is the Lie-algebra of  skew symmetric tensors
and $\mathfrak{sl}(3):=\{X\in\mathbb{R}^{3\times3}\;| \tr({X})=0\}$ is the Lie-algebra of traceless tensors. For $X\in {\rm PSym}(n)$,   $X=\sum_{i=1}^n \lambda_i \, N_i\otimes N_i$, where
$\lambda_i$  are the eigenvalues and $N_i$ are the eigenvectors of $X$,  we consider  $\log X=\sum_{i=1}^n \log \lambda_i (N_i\otimes N_i)$.
   {Here and i}n the following the superscript
$^T$ is used to denote transposition. For all  vectors $\xi,\eta\in\R^3$ we have the (dyadic) tensor product
$(\xi\otimes\eta)_{ij}=\xi_i\,\eta_j$. The set of positive real numbers is
denoted by $\R_+:=(0,\infty)$, while $\overline{\R}_+{:}=\R_+\cup \{\infty\}$.

Let us consider $W(F)$ to be the strain energy function of an elastic material in which $F=\nabla \varphi$ is the gradient of a deformation from a reference configuration
to a configuration in the Euclidean 3-space; $W(F)$ is measured per unit volume of the reference configuration. The domain of $W(\cdot)$ is ${\rm
GL}^+(n)$.  We denote by $C=F^T F$ the right Cauchy-Green strain tensor, by $B=F\, F^T$ the left Cauchy-Green (or Finger) strain tensor, by $U=\sqrt{F^T F}$ the right
stretch tensor, i.e., the unique element of ${\rm PSym}(n)$ for which $U^2=C$ and by  $V$ the  left stretch tensor, i.e., the unique element of ${\rm
PSym}(n)$ for which $V^2=B$. Here,  we are only concerned with rotationally symmetric functions (objective and isotropic), i.e.,
 $
 W(F)={W}(Q_1^T\, F\, Q_2) \ \forall \, F=R\,U=V R\in {\rm GL}^+(n),\ Q_1,Q_2,R\in{\rm SO}(n).
 $ We  denote by $S_1=D_F[W(F)]$ the first Piola-Kirchhoff stress tensor and by $\sigma=\frac{1}{J}\,  S_1\, F^T$ the Cauchy stress tensor. In this paper, the scalar product will be always denoted by $\langle \,,\, \rangle$, while ``$\cdot$'' will denote the multiplication with scalars or the multiplication of matrices. We will also use ``.'' to denote  by $D_F[W(F)].\, H$ the Fr\'echet derivative of the function $W(F)$ of   a tensor $F$ applied  to the tensor-valued increment $H$.

 Further  $\mu>0$ is the infinitesimal shear modulus,
$\kappa=\frac{2\mu+3\lambda}{3}>0$ is the infinitesimal bulk modulus with $\lambda$ the first Lam\'{e} constant, $C_e:=F_e^TF_e$ is the elastic strain tensor,  $U_e$ is the right elastic
stretch tensor, i.e. the unique element of ${\rm PSym}(n)$ for which $U_e^2=C_e$  and
\begin{align}
F=F_e\cdot  F_p
\end{align}
is the multiplicative decomposition of the deformation gradient \cite{lee1969elastic,neff2009notes,carstensen2002non,mielke2003energetic,mielke2002finite,NeffGhibaCompIUTAM}, while $C_p:=F_p^TF_p$ is the plastic  metric and  $U_p\in{\rm PSym}(n)$ is the plastic stretch,  $U_p^2=C_p$.

\subsection{The Hencky energy and elasto-plasticity}

 As hinted at above, the logarithmic strain space
description   is arguably the {simplest} algorithmic  approach to finite plasticity, suitable for the phenomenological description of isotropic
polycrystalline metals if the structure of geometrically linear theories {is} used with respect to the Lagrang{i}an logarithmic strain $\log C$. As deduced by Itskov
\cite{itskov2004application}, the logarithmic strain yields the most appropriate results in case of rigid-plastic material
subjected to simple shear.

 In isotropic finite strain computational   hyperelasto-plasticity \cite{bertram1999alternative,shutov2013analysis,shutov2008finite,reese2008finite,dettmer2004theoretical,Reese97a} the mostly used elastic energy is the quadratic Hencky logarithmic energy
\cite{Simo85,HutterSFB02,BatheEterovic,xiao2000consistent,bruhns1999self,henann2009large} (see also \cite{heiduschke1995logarithmic,
sansour2001dual,peric1992model,papadopoulos1998general,gabriel1995some,mosler2007variational}).
Among the works which use the Hencky strain in
elasto-plasticity we may also mention
\cite{simo1992algorithms,MieheApel,papadopoulos1998general,caminero2011modeling,peric1992model,peric1999new,dvorkin1994finite,geers2004finite}.
The  Hencky energy $W_{_{\rm H}}$ is the energy considered by the late  J.C. Simo (see Eq. (3.4), page 147, from \cite{simo1993recent} and also \cite{armero1993priori})
because
\begin{align}
W_{_{\rm H}}(F_e):&=\mu\,\|\dev_n\log \sqrt{F_e^TF_e}\|^2+\frac{\kappa}{2}\,[\tr(\log \sqrt{F_e^TF_e})]^2=\frac{\mu}{4}\,\|\dev_n\log
{F_e^TF_e}\|^2+\frac{\kappa}{8}\,[\tr(\log F^T F)]^2\\
&=\frac{\mu}{4}\,\|\dev_n\log {F_e^TF_e}\|^2+\frac{\kappa}{2}\,[\log (\det F)]^2.\notag
\end{align}
The Hencky energy  $W_{_{\rm H}}$ has the correct behaviour for extreme strains in the sense that $W(F_e)\rightarrow\infty$ as $\det
F_e\rightarrow0$ and, likewise $W(F_e)\rightarrow\infty$ as $\det F_e\rightarrow\infty$ \cite[page 392]{Simo98a}, but  it cannot be a polyconvex function of the deformation gradient.  However,  the model provides an excellent approximation for moderately large elastic
strains,  which is  superior to the usual Saint-Venant-Kirchhoff model of finite elasticity.   Several models of such a type have been considered in \cite{Neff_Wieners,krishnan2014polyconvex}. The decisive advantage of using the energy $W_{_{\rm H}}$ compared to other elastic
  energies stems from the fact that computational implementations of elasto-plasticity  \cite{gabriel1995some} based on the additive decomposition
  $\varepsilon=\varepsilon_e+\varepsilon_p$ in infinitesimal models\footnote{We need to be a little more specific. For the additive model in the format $\|\log C-\log C_p\|^2$ the complete systems of equations of the plastic flow rule are {\bf identical} to the infinitesimal additive model, while for the truly multiplicative model the return mapping algorithm is {\bf similar} to the infinitesimal case. } \cite{ebobisse2010existence,neff2009numerical,chleboun16extension}, can
  be used with nearly no changes also in isotropic finite  strain problems \cite[page 392]{Simo98a}. The computation of the elastic equilibrium at given plastic distortion $F_p$ suffers, however, under the well-known non-ellipticity  of $W_{_{\rm H}}$ \cite{Bruhns01,Balzani_Schroeder_Gross_Neff05,Neff_Diss00,Hutchinson82}. We know that
$W_{_{\rm H}}$ is Legendre-Hadamard elliptic in a  neighbourhood of the identity if $\lambda,\mu>0$, $\lambda_i\in[0.21162...,1.39561...]$  (see \cite{Bruhns01,Bruhns02JE}),
therefore $F\mapsto W_{_{\rm H}}(F)$ is  Legendre-Hadamard  elliptic for moderately large strains in the previous sense.

Moreover, the elastic Hencky energy has been shown \cite{RobertNeff,Neff_Osterbrink_Martin_hencky13} to have a fundamental differential geometric meaning, not shared by any other elastic energy, i.e.
\begin{align}
{\rm dist}^2_{{\rm geod}}\left( \frac{F}{(\det F)^{1/n}}, {\rm SO}(n)\right)&={\rm dist}^2_{{\rm geod,{\rm SL}(n)}}\left( \frac{F}{(\det F)^{1/n}}, {\rm
SO}(n)\right)=\|\dev_n \log U\|^2 \,,\\
{\rm dist}^2_{{\rm geod}}\left((\det F)^{1/n}\cdot \id, {\rm SO}(n)\right)&={\rm dist}^2_{{\rm geod,\mathbb{R}_+\cdot \id}}\left((\det F)^{1/n}\cdot \id,
\id\right)=|\log \det F|^2\,,\notag
\end{align}
where ${\rm dist}^2_{{\rm geod,\mathbb{R}_+\cdot \id}}$ and ${\rm dist}^2_{{\rm geod,{\rm SL}(n)}}$ are the canonical left invariant geodesic distances on
the Lie-group ${\rm SL}(n)$ and on the group $\mathbb{R}_+\cdot\id$, respectively (see
\cite{neff2013hencky}). For these investigations new mathematical tools had to be
discovered \cite{Neff_Nagatsukasa_logpolar13} also having consequences for the classical polar  decomposition.

\subsection{Additive metric plasticity}
The  Green-Naghdi additive plasticity models \cite{Naghdi65,green1971some} are well known (see e.g. \cite{casey1980remark,casey1984,Simo85}). They are based on the additive split  of the total Green-Saint-Venant strain $E=\frac{1}{2}(C-\id)$ into
\begin{align}
E=E_e+E_p,
\end{align}
where $E_e=\frac{1}{2}(C_e-\id)$ is the elastic Green-Saint-Venant strain and $E_p=\frac{1}{2}(C_p-\id)$ is the plastic strain. While $E_e=E-E_p$ was identified with elastic strain in the original work of Green
and Naghdi \cite{Naghdi65}, in later works \cite{green1971some} this identification
was dropped (see also \cite{lubliner2008plasticity}). This decomposition is  justified  as an approximation that is valid \cite{lubliner2008plasticity} when (i) small plastic
deformations are accompanied by moderate elastic strains, (ii) small elastic
strains are accompanied by moderate plastic deformations, or (iii) small
strains are accompanied by moderate rotations. A formulation based on elastic strain-measures like
\begin{align}
W(E-E_p):=\frac{\mu}{4}\, \|E-E_p\|^2+\frac{\lambda}{8}\, [\tr(E-E_p)]^2
\end{align}
is not even rank-one convex for zero plastic strain (this is the well-known deficiency of the Saint-Venant-Kirchhoff model \cite{Neff_Diss00,Raoult86}).

Another  finite plasticity  model using logarithmic strains is taking the additive  elastic Hencky energy \cite{miehe2009finite,miehe2011coupled,papadopoulos2001formulation,lu2004covariant,loblein2003application,schroder2002simple,schmidt2005some} in the format
\begin{align}
\widehat{W}_{_{\rm H}}(\log U-\log U_p)
&={\mu}\, \|\dev_3[\log U-\log U_p]\|^2+\frac{\kappa}{2}\,[\tr(\log U)]^2,\notag
\end{align}
as a starting point, in which plastic incompressibility $\tr(\log U_p)=0$ ($\det U_p=1$) is already included. This approach is based on the ad hoc introduction of the symmetric elastic strain tensor
$$E_e^{\log}:=E^{\log}-E_p^{\log}= \log U- \log U_p=\frac{1}{2}\, \log C-\frac{1}{2}\, \log C_p,$$
in the spirit of \cite{Naghdi65}.
An additive decomposition
of logarithmic strain has  been used in \cite{sansour2003viscoplasticity}  to construct a viscoplasticity theory. Note that the additive decomposition of the total strain is also considered in thermo-mechanics, see e.g.  \cite{casey1998elastic,lee2001theory}. Eisenberg et al. \cite{eisenberg1977observations} employ an additive decomposition of strain for thermo-plastic materials
and provide experimental correlation to theory. For numerical results see also \cite{ulz2009green,ulz2011finite}.

Despite appearance, we need to remark that this additive model has not much  in common with models based on the multiplicative decomposition, which will be shortly discussed in Section \ref{sectionrank}. Indeed, the independent plastic variable in the additive model is, in fact, $E_p^{\log}=\frac{1}{2}\,\log C_p=\log U_p$, which is determined to be trace free.  Moreover, the definition of $E_p^{\log}$ avoids any ambiguity of plastic rotation. Here, $C_p=U_p^2\in {\rm PSym}(3)$ takes only formally  the role of a local plastic  metric.
The flow rule will be an evolution equation in terms of  the independent variable $E_p^{\log}=\log U_p\in \mathfrak{sl}(3)\cap {\rm Sym}(3)$
\begin{align}\label{logcchi00}
\frac{\rm d}{{\rm d t}}[\log U_p]\in\, \Partial \Chi (\dev_3\Sigma) \qquad \Leftrightarrow\qquad \frac{\rm d}{{\rm d t}}[\log U_p]=\lambda_{\rm p}^+\, \frac{\dev_3\Sigma}{\|\dev_3\Sigma\|},
\end{align}
where $\Sigma=D\,[\widehat{W}_{_{\rm H}}(\log U-\log U_p)]$, $\Partial \Chi (\dev_3\Sigma)$ is the subdifferential  of the indicator function
$ \Chi (\dev_3\Sigma)$  of the convex elastic domain
$
\mathcal{E}_{\rm e}({\Sigma},{\frac{2}{3}}\, \sigma_{\textbf{y}}^2):=\left\{{\Sigma}\in {\rm Sym}(3)\, \Big| \, \|\dev_3{\Sigma}\|^2\leq {\frac{2}{3}}\, \sigma_{\textbf{y}}^2\right\},
$
 where $ \sigma_{\textbf{y}}$ is the yield limit and the plastic multiplier $\lambda^+_{\rm p}$ satisfies the  Karush–Kuhn–Tucker (KKT)-optimality constraints
\begin{align}\label{subgama}
\lambda^+_{\rm p}\geq 0, \qquad \Chi(\dev_3\Sigma)\leq0, \qquad \lambda^+_{\rm p}\, \Chi(\dev_3\Sigma)=0.
\end{align}
Note that since $U$ and $U_p$ do not commute, $\log U-\log U_p \neq \log(U\, U_p^{-1})$ in general. Having $\tr(\log U_p)=0$ consistent with the flow rule, it follows $\det U_p=1$. Moreover, $\log U_p\in {\rm Sym}(3)$ implies $U_p\in {\rm PSym}(3)$. The advantage of this model is that its structure w.r.t. to plasticity is identical to the infinitesimal model, while all proper invariances (objectivity and isotropy) of the geometrically nonlinear theory are retained. In contrast to multiplicative approaches the thermodynamic driving force is not the Eshelby tensor $\Sigma_{\rm E}$ and  this model gives decreasing shear stress in plastic simple shear which is physically unacceptable \cite{itskov2004application}.

Now, it could be argued that the reason for the former  deficiency is the dependence on the elastic strain $E_e^{\log}=\log U-\log U_p$, together with using the (already) non-elliptic elastic Hencky energy giving rise to a non-elliptic elastic formulation.
However, the additive logarithmic metric ansatz has another serious shortcoming which is the focus of this contribution and treated next.

\section{Rank one convexity and  plastic flow} \setcounter{equation}{0}\label{sectionrank}

Next, our goal is to understand the response of a plasticity formulation at given plastic distortion or plastic strain. Let us first consider small strain plasticity based on the additive decomposition of elastic strain
\begin{align}
{\rm sym} \nabla u=\varepsilon=\varepsilon_e+\varepsilon_p, \qquad \varepsilon_p\in {\rm Sym}(3).
\end{align}
The governing equations for isotropic perfect plasticity are
\begin{align}
{\rm Div}\, \sigma&=f,\qquad \quad
\sigma=2\, \mu\, ({\rm sym} \nabla u-\varepsilon_p)+\lambda\, ({\rm sym} \nabla u-\varepsilon_p)\, \id,\qquad
\frac{\rm d}{{\rm d} t}[{\varepsilon}_p]\in \Partial\Chi (\dev_3\Sigma_{\rm lin}),
\end{align}
 where  $\sigma$ is the Cauchy stress tensor and $\Sigma_{\rm lin}=D_{\varepsilon_p} W_{\rm lin}(\varepsilon-\varepsilon_p)$, $f$ is the body force, ${\rm Div} \sigma$ is a  vector having as components the divergence of the rows of the Cauchy stress tensor $\sigma$ and $\Partial \Chi$ is the subdifferential of the indicator function $\Chi$ of the convex elastic domain
 \begin{align}
 \mathcal{E}_{\rm e}(\Sigma_{\rm lin},\frac{2}{3}\, {\boldsymbol{\sigma}}_{\!\mathbf{y}}^2)=\left\{\Sigma_{\rm lin}\in{\rm Sym}(3) \big|\,\ \|\dev_3
\Sigma_{\rm lin}\|^2\leq\frac{2}{3}\, {\boldsymbol{\sigma}}_{\!\mathbf{y}}^2\right\}\subset  {\rm Sym}(3).
 \end{align}
  At frozen plastic straining $\varepsilon_p$, the elastically stored energy is given by
\begin{align}
W_{\rm lin}(\nabla u, \varepsilon_p)=\mu\, \|{\rm sym} \nabla u-\varepsilon_p\|^2+\frac{\lambda}{2}\, [\tr ({\rm sym} \nabla u-\varepsilon_p)]^2.
\end{align}
A simple observation is that the convexity of the function $W_{\rm lin}$ w.r.t. $\nabla u$ is not influenced by the plastic flow at all. Indeed, even the second derivative of $W_{\rm lin}$ (i.e. the elastic moduli) is unchanged. In other words, by plastic flow alone, the elastic unloading response is not touched upon.

Next, we turn to finite strain isotropic multiplicative plasticity. The system of equations can be written
\begin{align}
{\rm Div}\, S_1(F\, F_p^{-1})&=f,\quad 
S_1(F\, F_p^{-1})=D_F [W(F\, F_p^{-1})]=D_F[W(F_e)],\quad
-{F_p}\, \frac{\rm d}{{\rm d} t}[{F}_p^{-1}]&\in \Partial\Chi (\dev_3 \Sigma_{\rm E}),
\end{align}
where
$\Partial \Chi$ is the subdifferential of the indicator function $\Chi$ of the convex elastic domain
\begin{align}\mathcal{E}_{\rm e}(\Sigma_{\rm E},\frac{2}{3}\, \boldsymbol{\sigma}_{\!\mathbf{y}}^2)=\left\{\Sigma_{e}\in{\rm Sym}(3) \big|\,\ \|\dev_3
 \Sigma_{e}\|^2\leq\frac{2}{3}\, {\boldsymbol{\sigma}}_{\!\mathbf{y}}^2\right\}\subset  {\rm Sym}(3).
 \end{align}
  Here,   $\Sigma_{\rm E}$ is the elastic
 {Eshelby tensor}
$$
\Sigma_{\rm E}:=F_e^TD_{F_e}[W({F_e})]-W(F_e)\cdot \id,
$$
 driving the plastic evolution (see e.g. \cite{neff2009notes,maugin1994eshelby,cleja2000eshelby,cleja2003consequences,cleja2013orientational}) and $f$ is the body force.  As in the small strain case, rank-one convexity is preserved. However, the elastic moduli may decrease with evolving plastic flow. To substantiate  this claim let us  recall the following simple, yet fundamental observation:
\begin{lemma} \cite{HutterSFB02,Neff00b} {\rm (rank-one convexity and multiplicative decomposition)}\label{newlemma}
  If  the elastic energy $F\mapsto W(F)$ is rank-one convex,  it follows that the elasto-plastic formulation
\begin{align}
F\mapsto W(F,F_p):={W}(F\,F_p^{-1})={W}(F_e)
\end{align}
 remains rank-one convex w.r.t $F$ for all given plastic distortions $F_p$ . Hence, in the multiplicative setting, the elastic rank-one convexity is independent of the plastic flow. In multiplicative plasticity rank-one convexity is thus configuration independent.
 \end{lemma}

 The same constitutive invariance property is true for convexity, polyconvexity and quasiconvexity \cite{HutterSFB02,krishnan2014polyconvex}. Therefore,  the multiplicative approach is ideally suited as far as preservation of ellipticity properties for elastic unloading is concerned and marks a sharp contrast to additive logarithmic modelling frameworks, as will be seen.

 These observations, together with the experimental evidence that elastic unloading is  always seen to be  a stable process motivates us to postulate:
 \begin{postulate} {\rm [Stability at frozen plastic flow]}
 At frozen plastic flow the purely elastic response (elasticity with eigenstresses) should define a well-posed nonlinear elasticity problem in the sense that rank-one convexity is preserved. 
 \end{postulate}

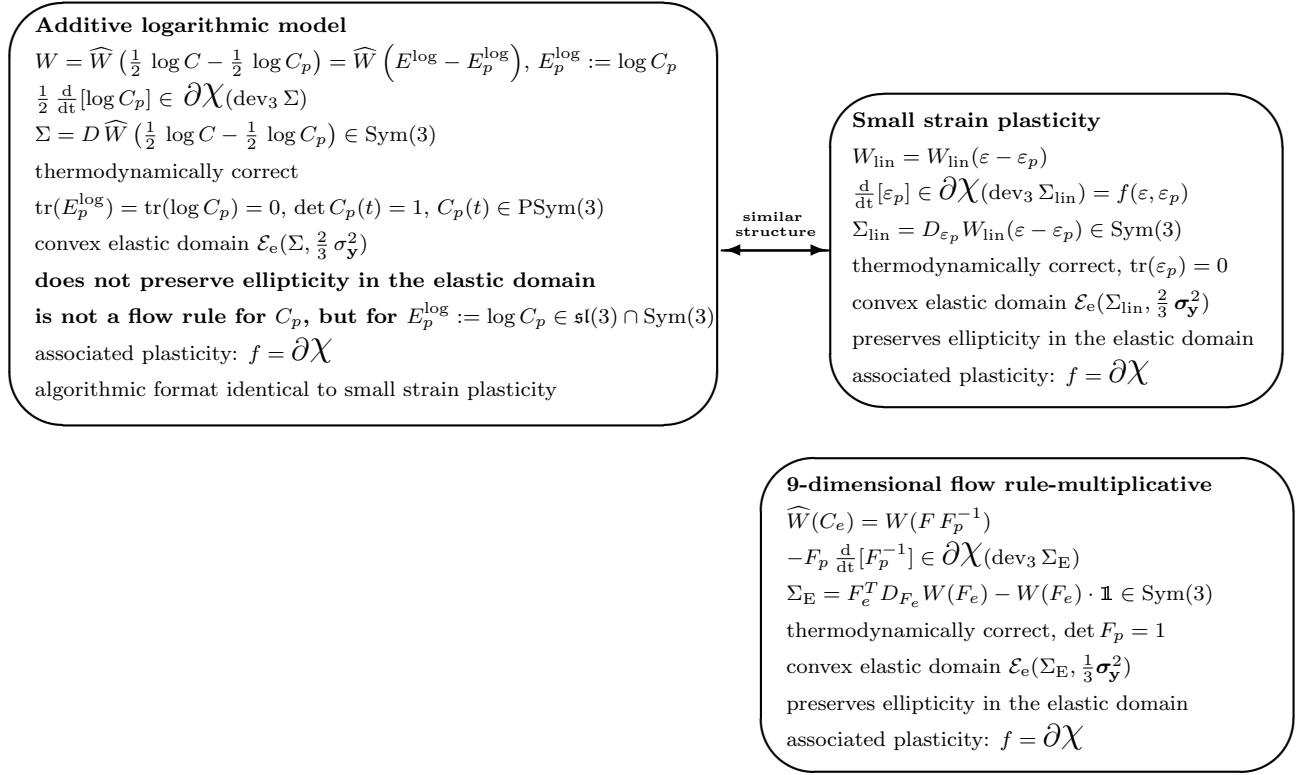
\begin{figure}
	\setlength{\unitlength}{0.97mm}
	\begin{center}
		\begin{picture}(10,90)
		\thicklines
		\put(-36,73){\oval(97,58)}
		\put(-81,98){\footnotesize{\bf Additive logarithmic model}}
		\put(-81,93){\footnotesize{$W=\widehat{W}\left(\frac{1}{2}\,\log C-\frac{1}{2}\,\log C_p\right)=\widehat{W}\left(E^{\log}-E_p^{\log}\right)$, $E_p^{\log}:=\log C_p$}}
		\put(-81,88){\footnotesize{$\frac{1}{2}\,\frac{\rm d}{{\rm d t}}[\log C_p]\in\, \Partial \Chi (\dev_3 \Sigma)$}}
		\put(-81,83){\footnotesize{$\Sigma=D\,\widehat{W}\left(\frac{1}{2}\,\log C-\frac{1}{2}\,\log C_p\right)\in{\rm Sym}(3)$}}
		\put(-81,78){\footnotesize{thermodynamically correct}}
		\put(-81,73){\footnotesize{$\tr(E_p^{\log})=\tr(\log C_p)=0$, $\det C_p(t)=1$, $C_p(t)\in{\rm PSym}(3)$}}
		\put(-81,68){\footnotesize{convex elastic domain $\mathcal{E}_{\rm e}({\Sigma},{\frac{2}{3}}\, \sigma_{\textbf{y}}^2)$}}
		\put(-81,63){\footnotesize{\bf does not preserve ellipticity in the elastic domain}}
		\put(-81,58){\footnotesize{\bf is not a flow rule for $C_p$, but for $ E_p^{\log}:=\log C_p\in\mathfrak{sl}(3)\cap{\rm Sym}(3)$}}
		\put(-81,53){\footnotesize{associated plasticity: $f= \Partial \Chi$}}
		\put(-81,48){\footnotesize{algorithmic format identical to small strain plasticity}}
		
		\put(13,68){\vector(1,0){14.5}}
		\put(28,68){\vector(-1,0){14.5}}
		\put(16,72){\bf \tiny{similar}}
		\put(15,70){\bf \tiny{structure}}

		\put(59,68){\oval(62,43)}
		\put(31,85){\footnotesize{\bf Small strain plasticity}}
		\put(31,80){\footnotesize{$W_{\rm lin}=W_{\rm lin}(\varepsilon-\varepsilon_p)$}}
		\put(31,75){\footnotesize{$\frac{\rm d}{{\rm d t}}[\varepsilon_p]\in\Partial \Chi(\dev_3 \Sigma_{\rm lin})=f(\varepsilon,\varepsilon_p)$}}
		\put(31,70){\footnotesize{$\Sigma_{\rm lin}= D_{\varepsilon_p} W_{\rm lin}(\varepsilon-\varepsilon_p)\in{\rm Sym}(3)$}}
		\put(31,65){\footnotesize{thermodynamically correct, $\tr(\varepsilon_p)=0$}}
		\put(31,60){\footnotesize{convex elastic domain $\mathcal{E}_{\rm e}({\Sigma_{\rm lin}},\frac{2}{3}\, {\boldsymbol{\sigma}}_{\!\mathbf{y}}^2)$}}
		\put(31,55){\footnotesize{preserves ellipticity in the elastic domain}}
		\put(31,50){\footnotesize{associated plasticity: $f= \Partial \Chi$}}
		
		\put(55,18){\oval(73,43)}
		\put(22,35){\footnotesize{\bf 9-dimensional flow rule-multiplicative}}
		\put(22,30){\footnotesize{$\widehat{W}(C_e)={W}(F\, F_p^{-1})$}}
		\put(22,25){\footnotesize{$-F_p\,\frac{\rm d}{{\rm d t}}[F_p^{-1}]\in\Partial \Chi(\dev_3 \Sigma_{\rm E})$}}
		\put(22,20){\footnotesize{$\Sigma_{\rm E}=F_e^T D_{F_e} {W}(F_e)-{W}(F_e)\cdot \id \in{\rm Sym}(3)$}}
		\put(22,15){\footnotesize{thermodynamically correct, $\det F_p=1$}}
		\put(22,10){\footnotesize{convex elastic domain $\mathcal{E}_{\rm e}({\Sigma_{\rm E}},\frac{1}{3}{\boldsymbol{\sigma}}_{\!\mathbf{y}}^2)$}}
		\put(22,5){\footnotesize{preserves ellipticity in the elastic domain}}
		\put(22,0){\footnotesize{associated plasticity: $f= \Partial \Chi$}}
		\end{picture}
	\end{center}
	\caption{An additive logarithmic model, additive small  strain plasticity.  All these models are associative, since all  flow rules are in the format $ \frac{\rm d}{\rm dt}[P]\, P^{-1}\in  -{\Large \partial}{\tiny} {\mbox{{\Large $\chi$}}}$ or $ \frac{\rm d}{\rm dt}[\varepsilon_p]\in  {\Large \partial}{\tiny} {\mbox{{\Large $\chi$}}}$ and for comparison the 9-dimensional flow rule based on the multiplicative decomposition.}\label{plastmodeldiagram2}
\end{figure}

\section{Additive logarithmic  plasticity does not preserve rank-one convexity during  plastic flow }\setcounter{equation}{0}

Let us start our argument by returning to a simplified one-dimensional situation for the sake of clarity.  Firstly, we observe that the classical quadratic Hencky strain energy
\begin{align}
t\mapsto (\log t)^2
\end{align}
is not convex (not rank-one convex) and therefore nothing can be gained by using the additive logarithmic anzatz
 \begin{align}
t\mapsto (\log t-\log s)^2,\qquad t\in (0,\infty)
\end{align}
where $s\in (0,\infty)$ is representing the plastic straining. However, the exponentiated Hencky energy
\begin{align}
t\mapsto e^{(\log t)^2}
 \end{align}
 is in fact a convex function of $t$ and the function
 \begin{align}\label{NR1}
 t\mapsto e^{(\log t-\log s)^2},\qquad t\in (0,\infty)
 \end{align}
 remains convex in $t$ for all $s\in (0,\infty)$  since $e^{(\log t-\log s)^2}=e^{\left(\log \frac{t}{s}\right)^2}$ is only a linear scaling  in the argument of  a convex function. Therefore in the one-dimensional setting the additive logarithmic framework preserves ellipticity (convexity).

 The last observation in \eqref{NR1}  motivated  us in a previous work \cite{NeffGhibaLankeit} to consider the following family of energies
\begin{align}\label{thdefHen}\hspace{-2mm}
 W_{_{\rm eH}}(F)=W_{_{\rm eH}}^{\text{\rm iso}}(\frac F{\det F^{\frac{1}{n}}})+W_{_{\rm eH}}^{\text{\rm vol}}(\det F^{\tel n}\cdot \id) = \left\{\begin{array}{lll}
\dd\frac{\mu}{k}\,e^{k\,\|{\rm dev}_n\log U\|^2}+\frac{\kappa}{2\widehat{k}}\,e^{\widehat{k}\,[(\log \det U)]^2}&\text{if}& \det\, F>0,\vspace{2mm}\\
+\infty &\text{if} &\det F\leq 0\,.
\end{array}\right.
\end{align}
We have called this the family of {\bf exponentiated Hencky energies}.  For the  two-dimensional situation $n=2$ and for  $\mu>0, \kappa>0$,   we have established that the functions $W_{_{\rm eH}}:\R^{2\times 2}\to \overline{\R}_+$ from the family of exponentiated Hencky type
energies are {\bf rank-one convex} \cite{NeffGhibaLankeit} for additional dimensionless material parameters $k\geq\dd\frac{1}{4}$ and $\widehat{k}\dd\geq \tel8$. Moreover, these energies are    {\bf polyconvex} \cite{NeffGhibaPoly,NeffGhibaSilhavy} and the corresponding   minimization problem
admits at least one solution.\medskip

 However, in the two-dimensional setting  the function
\begin{align}
F\mapsto W(F,U_p)=e^{\|\dev_2\log U-\dev_2\log U_p\|^2}=e^{\frac{1}{4}\|\dev_2\log C-\dev_2\log C_p\|^2}
 \end{align}
 looses ellipticity for essentially  non-coaxial plastic deformations, i.e., $C\, C_p^{-1}\neq C_p^{-1} C$. Otherwise, if we have the commutation relation $C\, C_p^{-1}=C_p^{-1} C$, then
 \begin{align}
 e^{\frac{1}{4}\|\dev_2\log C-\dev_2\log C_p\|^2}=e^{\frac{1}{4}\|\dev_2\log (C\, C_p^{-1})\|^2}
 \end{align}
 by the properties of the matrix logarithm and the
  function
 \begin{align}
 F\mapsto W(F,C_p)=e^{\frac{1}{4}\|\dev_2\log (C\, C_p^{-1})\|^2}
 \end{align}
  is  always Legendre-Hadamard elliptic  w.r.t. $F$ for  given plastic metric $C_p\in {\rm PSym}(2)$. Here, we have used that
 \begin{align}
 e^{\frac{1}{4}\|\dev_2\log (C\, C_p^{-1})\|^2}=e^{\frac{1}{4}\|\dev_2\log C_e\|^2}
 \end{align}
since the principal invariants of $C\, C_p^{-1}$ and $C_e$ are equal, and therefore the eigenvalues of $C\, C_p^{-1}$ are equal to the eigenvalues of $C_e$, see \cite{NeffGhibaPlasticity}. Then we may invoke Lemma \ref{newlemma} applied to $W(F_e)=e^{\frac{1}{4}\|\dev_2\log F_e^TF_e\|^2}$.

Therefore,  considering the family $W_{_{\rm eH}}$ of exponentiated Hencky energies, the main result of this paper  shows the following inacceptable feature: even if the elastic energy is everywhere rank-one convex as a function of $F$, i.e. $F\mapsto \widehat{W}(\log U)$ is rank-one convex, the new function $F\mapsto \widehat{W}(\log U-\log U_p)$ need not remain rank-one convex at some given plastic stretch $U_p$ (viz. $E_p^{\log}$).  This type of loss of ellipticity is relevant in elastic unloading problems at given plastic deformation. It naturally appears in computation of the elastic spring back. In other words we demonstrate that, relative to an initial natural configuration prior to the occurrence of further yielding, the elastic rank-one convexity property may be lost with development of plastic flow alone.  We note again that this type of degenerate response cannot occur in elasto plasticity based on the multiplicative decomposition of the deformation gradient into incompatible elastic and plastic distortion $F=F_e\cdot F_p$, thus giving more credit to these  latter type of models.

We consider now the isotropic exponentiated energy $F\mapsto W_{_{\rm eH}}^{\rm iso}(F)=e^{\|\dev_2\log U\|^2}$ corresponding to the material parameter $k=1$ and we provide the proof  that the new function $F\mapsto  \widehat{W}_{_{\rm eH}}^{\rm iso}(\log U-\log U_p)$ need not remain rank-one convex at some given plastic stretch $U_p$.

\begin{lemma}
The function $F\mapsto W(F)=e^{\|\dev_2\log U-\dev_2\log U_p\|^2}$ is not rank-one convex  for some given plastic stretch $U_p\in {\rm PSym}(2)$, while $F\mapsto W(F)=e^{\|\dev_2\log U\|^2}$ is  rank-one convex.
\end{lemma}
\begin{proof}
The rank-one convexity of the isotropic exponentiated energy $F\mapsto W_{_{\rm eH}}^{\rm iso}(F)=e^{\|\dev_2\log U\|^2}$ follows as a particular case of the result established in \cite{NeffGhibaLankeit}.

In order to prove that the function $F\mapsto W(F)=e^{\|\dev_2\log U-\dev_2\log U_p\|^2}$ is not rank-one convex  for some given  $U_p\in {\rm PSym}(2)$, it suffices to consider the elastic simple shear case. We choose the vectors $\eta, \xi\in \R^3$ so that
$
\eta=\left(\begin{array}{c}
      1 \\
0
    \end{array}\right),\quad \xi=\left(\begin{array}{c}
      0 \\
1
    \end{array}\right),\quad
\eta\otimes \xi =\left(
                   \begin{array}{cc}
                     0 & 1 \\
                     0 & 0
                   \end{array}
                 \right),
$
and we will be able to find a plastic stretch $U_p\in{\rm PSym}(2)$ such that the scalar function $h:\R\to \R$,
$$
h(t)=W(\id+t(\eta \otimes\xi))=e^{\|\dev_2\log \left\{\left[\id+t(\eta \otimes\xi)\right]^T\,\left[\id+t(\eta \otimes\xi)\right]\right\}-\dev_2\log U_p\|^2}
$$
is not convex as function of $t$. This is sufficient for loss of rank-one convexity. 
Thus we consider
$
F=\id+t(\eta \otimes\xi) =\left(
                   \begin{array}{cc}
                     1 & t  \\
                     0 & 1
                   \end{array}
                 \right)$.

\begin{figure}[h!]
\centering
\begin{minipage}[h]{0.9\linewidth}
\includegraphics[scale=0.7]{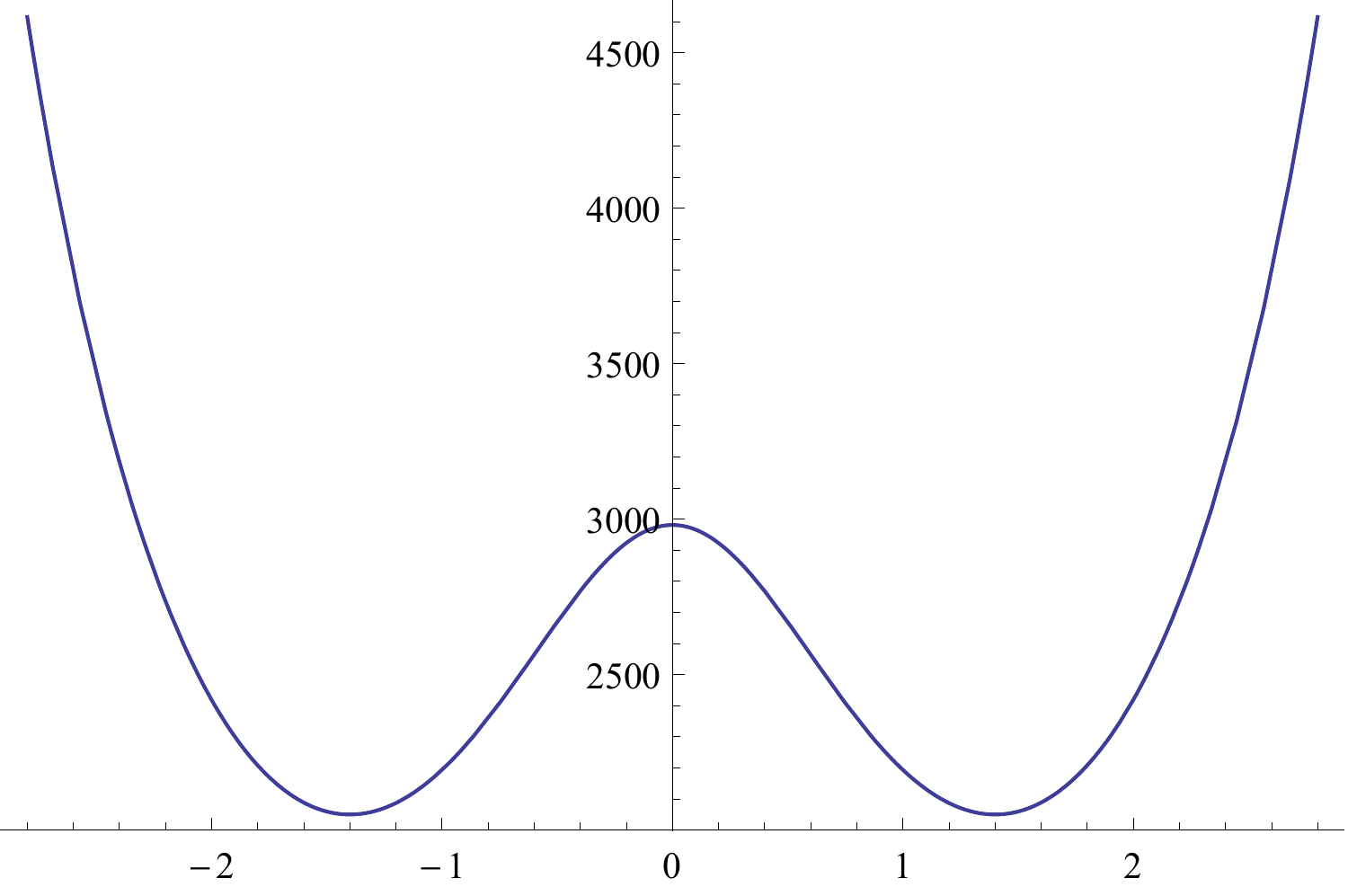}
\centering
\caption{\footnotesize{The graphical representation of the function $t\mapsto
		 h_{-2,0}(t)=e^{2\,\log^2\left[\frac{1}{2}(\sqrt{t^2+4}+t)\right]-8\,t\,\frac{\log \left[\frac{1}{2}(\sqrt{t^2+4}+t)\right]}{t^2+4}+8}
		$. }}
\label{corect-aditional}
\end{minipage}
\end{figure}

For this total  deformation  we have that the polar decomposition  $F=R\cdot U=V\cdot R$ into the right Biot stretch tensor $U=\sqrt{F^T F}$ of the deformation and the orthogonal polar factor $R$ is given by
\begin{align}
U=\frac{1}{\sqrt{t^2+4}}\left(
\begin{array}{cc}
2&t\\
t&t^2+2
\end{array}\right), \qquad
R=\frac{1}{\sqrt{t^2+4}}\left(
\begin{array}{cc}
2&t\\
-t&2
\end{array}\right).
\end{align}
Let us rewrite the right Biot stretch tensor $U$ in the following form
\begin{align}
U=\frac{\lambda_1}{\lambda_1 ^2+1}\,  \left(
	\begin{array}{cc}
	2 & \lambda_1 -\frac{1}{\lambda_1 } \vspace{1.2mm}\\
	\lambda_1 -\frac{1}{\lambda_1 } & \lambda ^2+\frac{1}{\lambda_1 ^2} \\
	\end{array}
	\right),
\end{align}
where $\lambda_1=\frac{1}{2}(\sqrt{t^2+4}+t)$ denotes the first eigenvalue of $U$.
Further, $U$ can be orthogonally diagonalized to
\begin{align}
U=Q\cdot \left(
\begin{array}{cc}
\frac{1}{\lambda_1}&0\\
0&\lambda_1
\end{array}\right)\cdot {Q^T},
\end{align}
where
\begin{align}
Q= \frac{\sqrt{\lambda_1}}{\sqrt{\lambda_1^2+1}}\,\left(
	\begin{array}{cc}
	-\sqrt{\lambda_1 } & \frac{1}{\sqrt{\lambda_1 }} \vspace{1.2mm}\\
	\frac{1}{\sqrt{\lambda_1 }} & \sqrt{\lambda_1 } \\
	\end{array}
	\right). 
\end{align}
 Hence, the principal logarithm of $U$ is
\begin{align}
\log U=Q\cdot \left(
\begin{array}{cc}
-\log\lambda_1&0\\
0&\log\lambda_1
\end{array}\right)\cdot Q^{T}=\frac{1}{\sqrt{t^2+4}}\cdot \left(
\begin{array}{cc}
-t \,\log\lambda_1&2\, \log\lambda_1\\
2\, \log\lambda_1&t\,\log\lambda_1
\end{array}\right).
\end{align}

We still need to determine a suitable plastic stretch $U_p\in {\rm SL}(2)$.
The general form of a matrix $U_p\in{\rm PSym}(2)$ such that  $\tr(\log U_p)=0$ is
\begin{align}
 \log U_p=\left(
\begin{array}{ccc}
a &b\\
b&-a
\end{array}\right)\quad
 \Rightarrow \quad  U_p=\exp\left(
\begin{array}{ccc}
a &b\\
b&-a
\end{array}\right).
\end{align}

 We  obtain that
 \begin{align}
 \|\dev_2\log U-\dev_2\log U_p\|^2&=\|\dev_2\log U\|^2-2\, \langle \dev_2\log U,\dev_2\log U_p\rangle+\|\dev_2\log U_p\|^2\\
 &=2\,\log^2\lambda_1-2\,\frac{\log \lambda_1}{t^2+4}(-2\, a\, t+4\, b)+2\, a^2+2\, b^2.\notag
 \end{align}
 With this representation we are able to disprove rank-one convexity of the function $F\mapsto e^{\|\dev_2\log U-\dev_2\log U_p\|^2}$. We have to check the convexity of the function
$$
 h_{a,b}(t)=W(\id+t(\eta \otimes\xi))=e^{2\,\log^2\left[\frac{1}{2}(\sqrt{t^2+4}+t)\right]-2\,\frac{\log \left[\frac{1}{2}(\sqrt{t^2+4}+t)\right]}{t^2+4}(-2\, a\, t+4\, b)+2\, a^2+2\, b^2}, \qquad  t\in \mathbb{R}.
$$

For our purpose, it is enough to find a matrix $U_p$, i.e. some values for $a,b\in \mathbb{R}$, such that the function $t\mapsto h_{a,b}(t)$ is not convex. For simplicity, let us choose $b=0$.
We remark from Figure \ref{corect-aditional} that for $b=0$ and $a=-2$,  i.e. for $U_p=\left(
\begin{array}{ccc}
e^{-2} &0\\
0&e^{2}
\end{array}\right)\in {\rm SL}(2)$,  the function
$$
t\mapsto h_{-2,0}(t)=e^{2\,\log^2\left[\frac{1}{2}(\sqrt{t^2+4}+t)\right]-8\,t
	\,\frac{\log \left[\frac{1}{2}(\sqrt{t^2+4}+t)\right]}{t^2+4}+8}, \qquad t\in \mathbb{R}.
$$
is not convex. In the following we prove this observation analytically.

Let us suppose that the function $h_{-2,0}$ is convex. On one hand, this assumption  implies that the first derivative  $t\mapsto h_{-2,0}'(t)$ is a monotone increasing function on $[0, \infty)$ and we would have
\begin{align}
h_{-2,0}'(t)\geq h_{-2,0}'(0) \qquad \forall t>0.
\end{align} 
On the other hand, since $h_{-2,0}(t)=h_{-2,0}(-t) \,\,\forall \, t\in\mathbb{R}$, we deduce that
$
h_{-2,0}'(t)=-h_{-2,0}'(-t) \ \text{for all}\ t\in\mathbb{R},
$
which implies that
$
h_{-2,0}'(0)=0.
$
Hence, we have obtained that
\begin{align}
h_{-2,0}'(t)\geq0 \qquad \forall t>0,
\end{align} which means that $h_{-2,0}$ is a monotone increasing function on $[0,\infty)$. However, this is not true, since
\begin{align}
h_{-2,0}(0)=e^8\approx 2980.96 \qquad \text{and}\qquad h_{-2,0}(1)\approx 2193.36.
\end{align}
Therefore, the assumption that the function $h_{-2,0}$ is convex leads  to a contradiction.

  In conclusion, the function $h_{-2,0}$ is not convex, which implies that for $U_p$ chosen as  above the energy $F\mapsto e^{\|\dev_2\log U-\dev_2\log U_p\|^2}$ is not rank-one convex, while $F\mapsto e^{\|\dev_2\log U\|^2}$ is rank-one convex \cite{NeffGhibaLankeit}. This final remark completes the proof.
\end{proof}

\section{Final remarks}

We have shown that the multiplicative plasticity models preserve ellipticity in purely elastic processes at frozen plastic variable provided that the initial elastic response is elliptic, see \cite{NeffGhibaPlasticity}.  Preservation of Legendre-Hadamard ellipticity is, in our view, a property which should be satisfied by any hyperelastic-plastic model since the elastically unloaded material specimen should respond reasonably under further purely elastic loading. However, the much used additive logarithmic model does not preserve  Legendre-Hadamard ellipticity in general. In Figure \ref{plastmodeldiagram2} we  summarize some properties of the additive logarithmic model and the multiplicative decomposition and we compare it with the small strain plasticity model.  In contrast, the formulation  based on $e^{\|\dev_2\log C_e\|^2}$ and the multiplicative decomposition remains always rank-one convex. Moreover, the change of a given FEM-implementation
of $W_{_{\rm H}}$ into $W_{_{\rm eH}}$ is nearly free of costs \cite{tanaka1979finite,masud1997finite,naghdabadi2012viscoelastic,heiduschke1996computational}.
 For more constitutive issues regarding the interesting properties of $W_{_{\rm eH}}$ we refer to \cite{NeffGhibaLankeit,montella2015exponentiated}.

\bibliographystyle{plain} 
\addcontentsline{toc}{section}{References}
\begin{footnotesize}

\end{footnotesize}

\end{document}